\documentclass[12pt, a4paper]{amsart}

\usepackage[english]{babel}
\usepackage{amsmath}
\usepackage{amssymb}
\usepackage{bm}
\usepackage{amscd} 
\usepackage{microtype} 
\usepackage{verbatim} 
\usepackage{color}  
\usepackage{tikz-cd}\usetikzlibrary{babel} 
\usepackage{enumitem}
\usepackage{mathtools} 
\usepackage{cite}
\usepackage{hyperref} 
\usepackage[initials,msc-links,backrefs]{amsrefs}
\usepackage[all]{xy}

\usepackage[OT2, T1]{fontenc}

\title[Gassmann triples with special cycle types]{Gassmann triples with special cycle types and applications}

 \author[H. Kammeyer]{Holger Kammeyer}
 \author[S. Kionke]{Steffen Kionke}
 
 \address{Mathematical Institute, Heinrich Heine University D{\"u}sseldorf, Germany}
 \email{holger.kammeyer@hhu.de}

 \address{Faculty of Mathematics and Computer Science, FernUniversit\"at in Hagen, Germany}
 \email{steffen.kionke@fernuni-hagen.de}
 
 \makeatletter
\@namedef{subjclassname@2020}{%
  \textup{2020} Mathematics Subject Classification}
\makeatother


\theoremstyle{plain}
\newtheorem{theorem}{Theorem}
\newtheorem{lemma}[theorem]{Lemma}
\newtheorem{corollary}[theorem]{Corollary}
\newtheorem{proposition}[theorem]{Proposition}

\theoremstyle{definition}
\newtheorem{definition}[theorem]{Definition}
\newtheorem*{definition*}{Definition}

\newtheorem*{observation*}{Observation}

\providecommand{\ignore}[1]{}

\providecommand{\Q}{\mathbb{Q}}

\providecommand{\C}{\mathbb{C}}

\DeclareMathOperator{\Fix}{Fix}
\DeclareMathOperator{\supp}{supp}

\newcommand*{\arXiv}[1]{ \href{http://www.arxiv.org/abs/#1}{arXiv:\textbf{#1}}}

\begin{document}

\begin{abstract}
We show that if one of various cycle types occurs in the permutation action of a finite group on the cosets of a given subgroup, then every almost conjugate subgroup is conjugate.  As a number theoretic application, corresponding decomposition types of primes effect that a number field is determined by the Dedekind zeta function.  As a geometric application, coverings of Riemannian manifolds with certain geodesic lifting behaviors must be isometric.
\end{abstract}

\maketitle

\section{Introduction}

Let \(G\) be a finite group.  Two subgroups \(H_1, H_2 \le G\) are called \emph{almost conjugate} if they intersect each conjugacy classe of \(G\) in the same number of elements.  In that case \((G; H_1, H_2)\) is called a \emph{Gassmann triple} and we say that a Gassmann triple is \emph{trivial} if \(H_1\) and \(H_2\) are honestly conjugate.  Gassmann triples naturally occur in number theory.  Let \(k_1\) and \(k_2\) be number fields and let \(K/\Q\) be a finite Galois extension containing both \(k_1\) and \(k_2\).  If \(H_1\) and \(H_2\) are the corresponding subgroups of \(G = \operatorname{Gal}(K/\Q)\), then \(\zeta_{k_1} = \zeta_{k_2}\) for the Dedekind zeta functions if and only if \((G; H_1, H_2)\) is a Gassmann triple.  In that case, \(k_1\) and \(k_2\) are called \emph{arithmetically equivalent}.  Of course \(k_1 \cong k_2\) if and only if \((G; H_1, H_2)\) is trivial.  For background and examples of Gassmann triples, we refer to~\cite{Klingen:similarities}.  So finding criteria under which a number field is determined by the zeta function translates to the question which conditions ensure that a Gassmann triple is trivial.  We offer the following result.

\begin{theorem} \label{thm:gassmann-triples}
  Let \(G\) be a finite group and let \(H_1 \le G\) be a subgroup of index \(n\).  Suppose that some element of \(G\) acts on \(G/H_1\) with cycle type
  \begin{enumerate}[label=(\roman*)]
  \item \label{item:twos} \((1,2,2, \ldots, 2)\), or
  \item \label{item:not-divide} \((a_1,\ldots, a_r, \ell)\) where \(\ell > 1\) is coprime to \(n a_1 \cdots a_r\), or
    \item \label{item:prime} \((a_1,\ldots, a_r, \ell)\) for a prime \(\ell\) that does not divide \(a_1 \cdots a_r\) and \(\ell\) is different from \(11\) and from \(\frac{q^k-1}{q-1}\) for all prime powers \(q\) and all \(k \geq 3\).
  \end{enumerate}
 Then every Gassmann triple \((G; H_1, H_2)\) is trivial. 
\end{theorem}

Let \(K/\Q\) be a Galois extension containing a number field \(k = \Q(a)\).  Each unramified rational prime \(p\) splits into distinct primes \(p = \mathfrak{p}_1 \cdots \mathfrak{p}_r\) in~\(k\) and the \emph{decomposition type} of \(p\), consisting of the residue degrees \((f_1, \ldots, f_r)\), coincides with the cycle type of the Frobenius automorphism of any prime \(\mathfrak{P} \mid p\) of \(K\) acting on the conjugates of \(a\).  Therefore:
  
\begin{theorem} \label{thm:prime-splitting}
  Let \(k_1\) be a number field of degree $n$ such that some unramified rational prime \(p\) has decomposition type \ref{item:twos}, \ref{item:not-divide}, or \ref{item:prime} above, or such that \(k_1\) has exactly one complex place or exactly one real place.  Then every number field \(k_2\) with \(\zeta_{k_1} = \zeta_{k_2}\) is isomorphic to \(k_1\).
\end{theorem}

Note that for \(q=2\) in \ref{item:prime}, primes of the form \(\frac{q^k-1}{q-1}\) are just the Mersenne primes.  So conjecturally, we have to exclude infinitely many primes \(\ell\) in \ref{item:prime}.  But \(\ell = 2, 3, 5, 17, 19, \ldots\) are the first few primes which work independently of \(n\) provided they do not divide any \(a_i\).  We point out that the special case \(\ell=2\) and correspondingly also the case of a unique complex place in Theorem~\ref{thm:prime-splitting}, was previously obtained by Chinburg--Hamilton--Long--Reid~\cite{Chinburg-et-al}*{Corollary~1.4} and independently by Stark~\cite{Stark}, a student of Perlis.

The primes in~\ref{item:prime} are genuine exceptions arising from Gassmann triples in \(\operatorname{PSL}(2,11)\) and \(\operatorname{GL}(k,q)\) respectively.  Using the LMFDB database~\cite{lmfdb} one can exhibit corresponding number fields. For instance, consider number fields~\(k_1\) and~\(k_2\) of degree~\(7\) defined by the polynomials
\begin{gather*}
  x^7 -7x +3,  \\
  x^7 +14x^4 -42x^2 -21x +9.
\end{gather*}
These number fields are not isomorphic but arithmeticaly equivalent.  The common Galois closure has the automorphism group \(G = \operatorname{GL}(3,2)\) of the Fano plane \(\mathbb{P}^2(\mathbb{F}_2)\) as Galois group and \(k_1\) and \(k_2\) correspond to the stabilizer subgroup \(H_1\) and \(H_2\) of a point and a line, respectively.  These subgroups are almost conjugate but not conjugate.  The group \(G\) has two conjugacy classes of order seven elements, each of size \(24\).  Since \(7 \nmid |H_1|=|H_2|\), all these 48 elements must act as full \(7\)-cycles on both \(G/H_1\) and \(G/H_2\).  Since \(48/168 = 6/21 \approx 0.29\), Chebotarev's density theorem says that roughly \(29\%\) of all unramified primes in \(k_1\) and \(k_2\) are inert, including the primes \(2\), \(5\), and \(11\) (see LMFDB database~\cite{lmfdb}).  Also note that Theorem~\ref{thm:prime-splitting} implies that all number fields of degree~7 which have an arithmetically equivalent sibling, like \(k_1\) and \(k_2\), must have signature \((3,2)\).

Similarly, according to~\cite{lmfdb}, the polynomials
\begin{gather*}
  x^{11} -\! 2x^{10} +\! 3x^9 +\! 2x^8 -\! 5x^7 +\! 16x^6 -\! 10x^5 +\! 10x^4 +\! 2x^3 -\! 3x^2 +\! 4x -\! 1, \\
  x^{11} - 2x^{10} + x^9 - 5x^8 + 13x^7 - 9x^6 + x^5 - 8x^4 + 9x^3 - 3x^2 - 2x + 1
\end{gather*}
define arithmetically equivalent number fields \(k_1\) and \(k_2\) of degree \(11\) whose common Galois closure has Galois group \(\operatorname{PSL}(2,11)\), a simple group of order \(660\).  It contains two conjugacy classes of elements of order \(11\), each of size \(60\).  Since \(11\) divides \(660\) only once, these \(120\) elements act as full \(11\)-cycles on the cosets of the subgroups corresponding to \(k_1\) and \(k_2\).  Since \(120/660 = 2/11 \approx 0.18\), about \(18\%\) of all unramified primes are inert, including the primes \(2\), \(5\), and \(11\).

\medskip
Apart from the number theoretic application, Gassmann triples also feature prominently in Sunada's construction~\cite{Sunada} of isospectral and iso-length manifolds.  To explain this, let \(M_0\) be a closed Riemannian manifold.  We adopt the terminology of \cite{Sunada}*{Section~4}:  A \emph{prime geodesic cycle} \(\mathfrak{p}\) in \(M_0\) is an oriented cycle represented by a closed geodesic in \(M_0\) which traces out the image exactly once.  The preimage of \(\mathfrak{p}\) under a finite sheeted covering map \(p \colon M \rightarrow M_0\), where \(M\) carries the metric induced from \(M_0\), decomposes into a collection \(\mathfrak{P}_1, \ldots, \mathfrak{P}_r\) of prime geodesic cycles in \(M\).  If \(c_i \colon S^1 \rightarrow M\) and \(c \colon S^1 \rightarrow M_0\) are (constant speed) representatives of \(\mathfrak{P}_i\)  and \(\mathfrak{p}\), respectively, then \(\pi \circ c_i (z) = c(z^{f_i})\) for \(z \in S^1 \subset \C\) and a unique positive integer \(f_i\) called the \emph{degree} of \(\mathfrak{P}_i\) over \(\mathfrak{p}\).

\begin{definition}
  We say that two finite coverings coverings \(p_1 \colon M_1 \rightarrow M_0\) and \(p_2 \colon M_2 \rightarrow M_0\) are \emph{arithmetically equivalent} if each prime geodesic cycle \(\mathfrak{p}\) of \(M_0\) has the same (unordered) tuple of degrees \((f_1, \ldots, f_r)\) in \(M_1\) and \(M_2\).
\end{definition}
  
Sunada~\cite{Sunada}*{Corollary to Theorem~2} showed that arithmetically equivalent coverings \(M_1\) and \(M_2\) have the same \emph{length spectrum}, meaning for every \(x \ge 0\), there exists a a bijection \(\phi_x \colon \mathcal{G}^1_x \rightarrow \mathcal{G}^2_x\) where
\[ \mathcal{G}^i_x = \{ \mathfrak{P} \text{ prime geodesic cycle in } M_i \colon \operatorname{length} (\mathfrak{P}) = x \}. \]
Using coverings corresponding to Gassmann triples with non-isomorphic subgroups, Sunada has given many examples of arithmetically equivalent coverings which are not isometric.  Theorem~\ref{thm:gassmann-triples} combined with Sunada's work, however, gives the following rigidity theorem. 

\begin{theorem} \label{thm:geodesic}
  Let \(p_1 \colon M_1 \rightarrow M_0\) be a connected \(n\)-sheeted covering of a closed Riemannian manifold \(M_0\).  Assume some prime geodesic cycle in \(M_0\) splits into prime geodesic cycles in \(M_1\) of degrees \ref{item:twos}, \ref{item:not-divide}, or \ref{item:prime} above.  If \(p_2 \colon M_2 \rightarrow M_0\) is a connected finite sheeted covering of \(M_0\) which is arithmetically equivalent to \(p_1\), then \(M_2\) is isometric to \(M_1\).
\end{theorem}

  In some cases, most notably if \(M_0\) is a rank one locally symmetric space, the length spectrum and the eigenvalue spectrum of the Laplace--Beltrami operator on functions determine each other so that arithmetically equivalent coverings are also isospectral~\cite{Sunada}*{Section~4}, see also \citelist{\cite{Huber} \cite{deVerdiere}}.  On the other hand, examples of Ikeda~\cite{Ikeda} show that iso-spectral manifolds must not be arithmetically equivalent which prevents us from concluding a pure ``spectral rigidity theorem'' from our results.  Nonetheless, H.\,Pesce~\cite{Pesce} has proven that ``generically'', isospectral manifolds do arise from Gassmann triples.  In that case the occurrence of geodesic cycle types~\ref{item:twos}, \ref{item:not-divide}, or \ref{item:prime} implies that the manifolds are in fact isometric (``one can hear the shape of the drum''.)  We remark that there is also a Sunada type construction of isospectral graphs using Gassmann triples~\cite{Brooks}*{Theorem~1.1}.  For graphs, however, a result of Brooks \cite{Brooks}*{Theorem 0.2} only gives an asymptotic converse.
  
  \medskip
  
  As a final application, it follows that the assumption in \cite{Kammeyer-Spitler}*{Theorem~1\,(ii)} that the number field \(k\) with precisely one real place should be locally determined is always satisfied thanks to Theorem~\ref{thm:prime-splitting}.  So the commensurability class of a Chevalley group of type \(A_n\) or \(C_n\) over a number field with precisely one real place is determined by the commensurability class of the profinite completion unless it has a proper Grothendieck subgroup.  In fact, this question was the starting point for the paper at hand.
  
  \medskip
  
  Let us outline the proof of Theorem~\ref{thm:gassmann-triples}.  To obtain part~\ref{item:twos} of the theorem, we have to conclude the triviality of Gassmann triples from the occurrence of a \emph{point involution}, meaning an involution with exactly one fixed point.  To do so, the main auxiliary result is that the permutation character of a Gassmann triple can detect whether two point involutions have the same fixed point.  This will be explained in Section~\ref{sec:point-involutions}.  To show parts~\ref{item:not-divide} and~\ref{item:prime} of Theorem~\ref{thm:gassmann-triples}, first note that by taking powers of the permutation, we can assume \(a_1 = \cdots = a_r = 1\).  So we have to conclude the triviality of Gassmann triples from the occurrence of certain cycles of prime length.  To do so, we need to dive a bit deeper into the theory of permutation groups.  In particular, we will show an induction principle saying that if a transitive permutation action generated by cycles only produces trivial Gassmann triples, then one has the same conclusion if this action occurs as a block of a permutation action such that all elements outside the block are moved by cycles.  The excluded primes \(l\) in~\ref{item:prime} result from a theorem of Feit.  It asserts that if \((G; H_1, H_2)\) is a Gassmann triple such that the permutation actions of \(G\) on \(G/H_1\) and \(G/H_2\) are doubly transitive and contain a full cycle, then \(G\) has a faithful representation as semilinear projective transformations over a finite field (this involves the classification of finite simple groups).  All this will be explained in Section~\ref{sec:cycles}.  For completeness, we conclude Theorems~\ref{thm:prime-splitting} and~\ref{thm:geodesic} in a short Section~\ref{sec:conclusion}.

  \medskip
  This work was financially supported by the German Research Foundation via the Research Training Group ``Algebro-Geometric Methods in Algebra, Arithmetic, and Topology'', DFG 284078965, and via the Priority Program ``Geometry at Infinity'', DFG 441848266.

  \section{Permutation groups with point involutions} \label{sec:point-involutions}

  In this section, we prove Theorem~\ref{thm:gassmann-triples} if condition~\ref{item:twos} is satisfied.  We begin with some standard facts on permutation actions.  By a \emph{\(G\)-set}, we mean a permutation representation of $G$ on a finite set $\Omega$.  Every $G$-set $\Omega$ gives rise to a linear representation of $G$ on $\Q[\Omega]$.  The isomorphism class of the linear representation is uniquely determined by the character
\[
	\chi_\Omega(g) = |\Fix(g,\Omega)|
\]
where $\Fix(g,\Omega)$ denotes the set of $g$-fixed points in $\Omega$.  The $G$-set~$\Omega$, however, is in general not determined up to isomorphism by the character~$\chi_\Omega$.
\begin{definition}
We say that two $G$-sets $\Omega, \Omega'$ are \emph{Gassmann equivalent}, if $\chi_\Omega = \chi_{\Omega'}$.  A $G$-set $\Omega$ will be called \emph{Gassmann solitary}, if every $G$-set that is Gassmann equivalent to $\Omega$ is already isomorphic to $\Omega$.
\end{definition}

It is easy to see that for a \emph{Gassmann triple} \((G; H_1, H_2)\), the $G$-sets $G/H_1$ and $G/H_2$ are Gassmann equivalent.  Conversely, if \(\Omega\) and \(\Omega'\) are Gassmann equivalent \emph{transitive} \(G\)-sets, then for every \(x \in \Omega\) and \(x' \in \Omega'\), the triple \((G; \operatorname{Stab}^\Omega_G(x), \operatorname{Stab}_G^{\Omega'}(x'))\) consisting of the stabilizer subgroups of \(x\) and~\(x'\) is a Gassmann triple.  A proof for both statements can be found in~\cite{Klingen:similarities}*{Theorem~1.3, p.\,77}.  It is moreover clear that a Gassmann triple \((G; H_1, H_2)\) is trivial if and only if $G/H_1$ and $G/H_2$ are isomorphic \(G\)-sets.

Let us now fix a finite set \(\Omega\).  A permutation \(\sigma \in \operatorname{Sym}(\Omega)\) is called a \emph{point involution} if \(\sigma^2 = \operatorname{id}\) and if \(\sigma\) fixes a unique point \(x \in \Omega\).

\begin{proposition} \label{prop:odd-even}
  Let \(\sigma, \tau \in \operatorname{Sym}(\Omega)\) be point involutions.  Then \(\sigma\) has the same fixed point as \(\tau\) if and only if \(\sigma \tau\) fixes an odd number of points.
\end{proposition}

\begin{proof}
  First suppose \(\sigma\) and \(\tau\) have the same fixed point \(x \in \Omega\).  Then \(x\) is a fixed point of \(\sigma \tau\) as well.  If \(y \in \Omega\) is a fixed point of \(\sigma \tau\) different from \(x\), then \(\sigma(y) = \tau(y)\) so \(\sigma \tau (\sigma(y)) = \sigma \tau (\tau(y)) = \sigma (y)\) which means \(\sigma(y)\) is another fixed point of \(\sigma \tau\) and \(\sigma(y)\) is different from \(y\) because \(x\) is the only fixed point of \(\sigma\).  This shows that the fixed points of \(\sigma \tau\) different from \(x\) come in pairs, hence the total number of fixed points of \(\sigma \tau\) is odd.

  Conversely, suppose \(\sigma(x) = x\) and \(\tau(y) = y\) with \(x \neq y\).  Then \(\sigma \tau(y) = \sigma(y) \neq y\) because \(x\) is the only fixed point of \(\sigma\).  Similarly, \(\tau(x) \neq x\) because \(y\) is the only fixed point of \(\tau\).  So \(\sigma \tau(x) \neq x \) because \(x\) is the only point that \(\sigma\) sends to \(x\).  For the remaining points, the same argument as above shows that fixed points of \(\sigma \tau\) come in pairs, hence the total number of fixed points of \(\sigma \tau\) is even.
\end{proof}

\begin{proposition}
Let \(G\) be a finite group and let \(\Omega\) be a transitive $G$-set.  Assume that some element of \(G\) acts as a point involution on \(\Omega\), then $\Omega$ is Gassmann solitary.
\end{proposition}

\begin{proof}
   Let $\Omega'$ be a $G$-set which is Gassmann equivalent to $\Omega$.
  Let \(x \in \Omega\).  Since some element \(\sigma \in G\) acts as a point involution on \(\Omega\) and \(G\) acts transitively on \(\Omega\), some conjugate \(\sigma_x\) of \(\sigma\) acts as a point involution with \(\sigma_x (x) = x\).  Since \(\chi_{\Omega} =  \chi_{\Omega'}\), we may conclude from \(\chi_\Omega(\sigma_x) = 1\)  that the element \(\sigma_x\) acts as a point involution on \(\Omega'\), too.  Hence there exists a unique element \(y \in \Omega'\) with \(\sigma_x(y) = y\) and we set \(f(x) = y\).

  We show that \(f\) is well-defined.  Suppose \(\tau \in G\) is another point involution fixing \(x\).  Then by Proposition~\ref{prop:odd-even}, the composition \(\sigma \tau\) acts on \(\Omega\) with \(\chi_\Omega(\sigma \tau)\) odd.  Hence \(\chi_{\Omega'}(\sigma \tau)\) is also odd.  Applying the proposition again shows that \(\sigma\) and \(\tau\) have the same fixed point in \(\Omega'\).  Interchanging moreoever the roles of \(\Omega\) and \(\Omega'\), we obtain an inverse of \(f\), so \(f\) is a well-defined bijection.

  It remains to show that \(f\) is \(G\)-equivariant.  To this end, let \(x \in \Omega\) and \(\tau \in G\).  Setting \(f(x) = y\), we have \(\sigma x = x\) and \(\sigma y = y\) for some \(\sigma \in G\) acting as a point involution both on \(\Omega\) and \(\Omega'\).  It follows that \(\tau \sigma \tau^{-1}\) is a point involution fixing \(\tau(x) \in \Omega\) and \(\tau(y) \in \Omega'\).  Thus \(f(\tau(x)) = \tau(y) = \tau(f(x))\).
\end{proof}

This last proposition clearly implies Theorem~\ref{thm:gassmann-triples} if condition~\ref{item:twos} holds true:  If the \(G\)-action on \(G/H_1\) has a point involution and \((G; H_1, H_2)\) is a Gassmann triple, we obtain a \(G\)-equivariant bijection \(f \colon G/H_1 \rightarrow G/H_2\), and if, say, \(f(H_1) = gH_2\), then \(g^{-1}H_1g = H_2\).

  \section{Permutation groups with cycles} \label{sec:cycles}
  
  In this section, we prove Theorem~\ref{thm:gassmann-triples} if condition~\ref{item:not-divide} or~\ref{item:prime} holds true.  For a subset \(S \subseteq \mathrm{Sym}(\Omega)\), we set \(\Fix(S,\Omega) = \bigcap_{\sigma \in S} \Fix(\sigma,\Omega)\).
  
\begin{definition}
We say that a $G$-set $\Omega$ admits a \emph{fixed point detector}, if there is an element $g \in G$ such that  $\Fix(g,\Omega) = \Fix(G,\Omega)$.
\end{definition}

It is well-known that every transitive $G$-set admits a fixed point detector~\cite{Wielandt}*{3.11} which in this case is just an element acting without fixed points.

\begin{lemma}\label{lem:fpd}
Every permutation group generated by a set of cycles admits a fixed point detector.
\end{lemma}

\begin{proof}
Let $G \subseteq \mathrm{Sym}(\Omega)$ be generated by a set $S$ of cycles.
Decompose $\Omega = F \cup \bigcup_{i=1}^k \mathcal{O}_i$ into the set of fixed points $F$ and a disjoint union of non-trivial $G$-orbits.
Every cycle $\sigma \in S$ is supported in exactly one orbit $\mathcal{O}_i$. This implies that $G$ decomposes as a direct product $G_1\cdot G_2\cdots G_k$ where $G_i$ is generated by the cycles supported in $\mathcal{O}_i$. Since $G_i$ acts transitively on $\mathcal{O}_i$ there is an element $g_i \in G_i$ without fixed points. We deduce that the element $g = g_1\cdots g_k \in G $ has $F = \Fix(g,\Omega)$.
\end{proof}
The existence of a fixed point detector immediately implies:
\begin{corollary}\label{cor:existence-of-fixed}
If $G \subseteq \mathrm{Sym}(\Omega)$ is generated by cycles, then 
$G$ has exactly $\min_{g\in G} \Fix(g,\Omega)$ fixed points in $\Omega$. 
\end{corollary}
%
%
\ignore{
\begin{corollary}\label{cor:fixedpointscontained}
Let $S_1,S_2 \subseteq \mathrm{Sym}(\Omega)$ be two sets of cycles. 
Then $\Fix(S_1,\Omega) \subseteq \Fix(S_2,\Omega)$ if and only if 
\[
	|\Fix(g,\Omega)| \geq |\Fix(S_1,\Omega)|
\]
for all $g \in \langle S_1 \cup S_2 \rangle$.
\end{corollary}
\begin{proof}
Let $F= \Fix(S_1,\Omega)$.
Assume that $F \subseteq \Fix(S_2,\Omega)$. Then $G = \langle S_1, S_2 \rangle$ fixes the points fixed by $S_1$ and hence $|\Fix(g,\Omega)| \geq |\Fix(S_1,\Omega)|$ holds for all $g \in G$.

Conversely, if $|\Fix(g,\Omega)| \geq |F|$ for all $g \in G$, then $G$ has at least $|F|$
 fixed points, which are necessarily the points in $F$. In particular, every element in $S_2$ fixes the points in $F$.
\end{proof}
}
Let $\Omega$ be a $G$-set.
A \emph{block} $B \subseteq \Omega$ is a non-empty subset such that for all $g \in G$ the sets $gB$ and $B$ are either equal or disjoint.
We denote the setwise stabilizer of $B$ by $G_{\{B\}}$.

\begin{proposition} \label{prop:solitary-using-blocks}
Let $G$ be a finite group and let $\Omega$ be a transitive $G$-set.
Assume that there is a set of cycles $S \subseteq G$ such that $B = \Fix(S,\Omega)$ is a block
and such that the subgroup $C_{\{B\}}\subseteq G_{\{B\}}$ generated by all cycles in \(G_{\{B\}}\) acts transitively on $B$.
If the action of $G_{\{B\}}$ on $B$ is Gassmann solitary, then
 $\Omega$ is Gassmann solitary.
\end{proposition}

\begin{proof}
Assume that $\Omega'$ is a $G$-set which is Gassmann equivalent to $\Omega$.
A cycle $\tau$ of length $\ell$ can be recognized using the character: It satisfies $\chi_\Omega(\tau^j) = n-\ell$ for all $j$ that are not multiples of $\ell$ and $\chi_\Omega(\tau^{m\ell}) = n$.  In particular, every element of $G$ which acts like a cycle on $\Omega$ acts like a cycle of the same length on $\Omega'$.  We define $B' = \Fix(S,\Omega')$.  We note that by Corollary~\ref{cor:existence-of-fixed}, applied to the faithfully acting quotient of \(G\), we have $|B| = |B'|$.  The crucial observation is that the cycles in $S$ allow us to describe $H = G_{\{B\}}$. Indeed, $g \in G$ lies in $H$ or outside \(H\) if and only if $\langle S \cup g S g^{-1} \rangle$ has \(|B|\) fixed points or no fixed points, respectively.  By Corollary~\ref{cor:existence-of-fixed}, this condition can be checked using \(\chi_\Omega\), so that $G_{\{B\}} = H =  G_{\{B'\}}$ and \(B'\) is a block, too.

Now we show that the character of the action of $H$ on $B$ can be computed from \(\chi_\Omega\).  A non-trivial cycle $\tau$ is contained in $B$ if and only if $\tau$ commutes with all elements in $S$ and $\langle \tau, S \rangle$ has less than $|B|$ fixed points.  This implies that also \(C_{\{B\}} = C_{\{B'\}}\). Let
\[ K = \langle\{ gtg^{-1} \mid g \not\in H, \ t \in C_{\{B\}}\}\rangle. \]
Then $K$ is a group generated by cycles whose orbits in $\Omega$ disjoint from \(B\) are exactly the sets $gB$ different from $B$. Similarly, the $K$-orbits on $\Omega'$ disjoint from \(B'\) are the sets $gB'$ different from $B'$.  Let $h \in H$ and let $f$ be the number of $h$-fixed points in $G/H$.  We define $E_h = \{(x,k) \in \Omega\times K\mid hkx = x\}$. Then 
\begin{align*}
	\sum_{k \in K} \chi_{\Omega}(hk) &= \sum_{(x,k) \in E_h} 1 \\
	&= \sum_{x \in B} |\{k \in K \mid hkx  = x\}| + \sum_{x \not\in B} |\{k \in K \mid hkx  = x\}| \\
	&= |K| \cdot |\Fix(h,B)| + \sum_{x \not\in B} |\{k \in K \mid hkx  = x\}|.
\end{align*}
If $x$ lies in a block $gB \neq B$ that is not fixed by $h$, then  $|\{k \in K \mid hkx  = x\}| = 0$.  On the other hand, if $hgB =gB$, then, since $K$ acts transitively on $gB$, there are exactly $|K_x|$ elements $k \in K$ with $hkx=x$. Using $|K/K_x| = |B|$ 
we deduce \[ \sum_{k \in K} \chi_{\Omega}(hk) = |K| \cdot |\Fix(h,B)| +  |K| \cdot (f-1). \]
Hence $|\Fix(h,B)|$ is determined by $\chi_{\Omega}$.  By assumption, the action of $H$ on $B$ is Gassmann solitary and so $B$, $B'$ are isomorphic as $H$-sets.  In particular, there are two points $b \in B$ and $b' \in B'$ which have the same stabilizer in $H$ and since $B$ and \(B'\) are blocks, these stabilizers agree with the stabilizers in $G$.
\end{proof}

For $B = \{x\}$ we conclude:

\begin{corollary}\label{lem:solitary-using-fixed-points}
Let $G$ be a finite group and let $\Omega$ be a transitive $G$-set.
If there is some $x \in \Omega$ and a set of cycles $S \subseteq G$ such that
\[
	\{x\} = \Fix(S,\Omega),
\]
then $\Omega$ is Gassmann solitary.
\end{corollary}

Recall that a \(G\)-set \(\Omega\) is called \emph{primitive} if it is transitive and has only the trivial blocks \(\{x\}\) for \(x \in \Omega\) and \(\Omega\).

\begin{proposition} \label{prop:primitive-and-cycles}
Let $G$ be a finite group and let $\Omega$ be a primitive $G$-set of degree $n$.
If some element of $G$ acts as a cycle of length $\ell$ with $1 < \ell < n$, then $\Omega$ satisfies the assumption of Corollary \ref{lem:solitary-using-fixed-points} and  is Gassmann solitary.
\end{proposition}

\begin{proof}
 Let $\sigma \in G$ denote a cycle of length $\ell$ with $1< \ell< n$.  We denote by $\Delta = \Omega \setminus \supp(\sigma) = \Fix(\sigma, \Omega)$ the complement of the support of $\sigma$.  We verify that the assumption of Corollary  \ref{lem:solitary-using-fixed-points} is satisfied.  We fix some $x \in \Delta$, i.e., a point which is fixed by $\sigma$.  The action is primitive, so Rudio's argument~\cite{Wielandt}*{8.1} shows that for all $y \in \Delta$  with $y \neq x$ there is an element $g\in G$ with $gy \in \supp(\sigma)$ and $gx \in \Delta$.  Therefore the cycle $\sigma_y = g^{-1}\sigma g$ fixes $x$ but moves $y$.  In particular, $x$ is the only point fixed by all elements in $S = \{\sigma\} \cup \{\sigma_y \mid y \in \Delta\setminus\{x\} \; \}$.  This proves the claim.
\end{proof}

\begin{proposition} \label{prop:coprime-primitive}
Let $G$ be a finite group and let $\Omega$ be a transitive $G$-set of degree $n$.
If $G$ is generated by cycles whose length \(\ell\) is either
\begin{enumerate}[label=(\arabic*)]
\item\label{it:coprime} coprime to $n$ or
\item\label{it:prime} a prime,
\end{enumerate}
then $\Omega$ is primitive.
\end{proposition}

\begin{proof}
Assume that $B \subsetneq \Omega$ is a nonempty block.
Since the action of $G$ is transitive and $G$ is generated by cycles, there is a cycle $\sigma$ such that $\supp(\sigma)$ intersects $B$ and $\Omega\setminus B$. Suppose that there is some $b \in B$ that is fixed by $\sigma$. Take $c \in B \cap \supp(\sigma)$, then $B$ contains $\sigma^j(c)$ for all $j$ and hence contains $\supp(\sigma)$. This contradicts our choice of $\sigma$. It follows that $B$ is contained in $\supp(\sigma)$. But then $|B|$ is a common divisor of $n$ and $\ell$ (compare~\cite{Wielandt}*{Exercise~6.5}.)  In case \ref{it:coprime}, we deduce $|B| = 1$. In case \ref{it:prime}, we have $|B| = 1$ or $|B| = \ell$. The second option $|B| = \ell$ implies that $B = \supp(\sigma)$ which contradicts our choice of $\sigma$. So in both cases $B$ is a singleton and $\Omega$ is a primitive $G$-set.
\end{proof}

We are now prepared to prove Theorem~\ref{thm:gassmann-triples} if condition~\ref{item:not-divide} holds true.  Note that the \(\operatorname{lcm}(a_1, \ldots, a_r)\)-th power of the permutation in~\ref{item:not-divide} provides a permutation of cycle type \((1, \ldots, 1, \ell)\) so that case~\ref{item:not-divide} of Theorem~\ref{thm:gassmann-triples} follows from the following proposition.

\begin{proposition}
Let $G$ be a finite group and let $\Omega$ be a transitive $G$-set of degree $n$.  Suppose that some $g \in G$ acts as a cycle of length $\ell > 1$, where $\ell$ is coprime to $n$. Then $\Omega$ is Gassmann solitary.
\end{proposition}

\begin{proof}
Let $N$ be the normal subgroup generated by all cycles of length $\ell$. Then $\Omega$ decomposes into disjoint $N$-orbits $\Omega = \mathcal{O}_1 \cup \dots \cup \mathcal{O}_r$ and $N$ as a direct product $N_1 \cdot N_2 \cdots N_r$  where $N_i$ is the subgroup generated by $\ell$-cycles in $\mathcal{O}_i$. Since the action is transitive, we have $t = |\mathcal{O}_1| = \dots = |\mathcal{O}_r|$ and $t$ divides $n$.  The action of $N_i$ on $\mathcal{O}_i$ is generated by $\ell$-cycles and  $\ell$ is coprime to $t$, so that $N_i$ acts primitively by Proposition~\ref{prop:coprime-primitive}.  In particular, by Proposition~\ref{prop:primitive-and-cycles}, each $N_i$ satisfies the assumption of Corollary \ref{lem:solitary-using-fixed-points}. Moreover, each $N_j$ ($j\neq i$) is generated by cycles and thus the group $G$ satisfies the assumption of Corollary \ref{lem:solitary-using-fixed-points}.  It follows that $\Omega$ is Gassmann solitary.
\end{proof}

Finally, we prove Theorem~\ref{thm:gassmann-triples} if condition~\ref{item:prime} holds true:

\begin{proposition}
Let \(G\) be a finite group and let \(\Omega\) be a transitive \(G\)-set.  Suppose that some \(g \in G\) acts as a cycle of prime length \(\ell\), where \(\ell\) is different from \(11\) and from \(\frac{q^k-1}{q-1}\) for prime powers \(q\) and \(k \geq 3\).  Then \(\Omega\) is Gassmann solitary.
\end{proposition}

\begin{proof}
  We argue similarly as before.  So $N$ denotes the normal subgroup generated by all \(\ell\)-cycles, $\Omega = \mathcal{O}_1 \cup \dots \cup \mathcal{O}_r$ is the disjoint decomposition into \(N\)-orbits, and $N$ is the direct product $N_1 \cdot N_2 \cdots N_r$  where $N_i$ is generated by \(\ell\)-cycles in $\mathcal{O}_i$.  Again $t = |\mathcal{O}_1| = \dots = |\mathcal{O}_r|$ because the \(G\)-action on \(\Omega\) is transitive.  The action of $N_i$ on $\mathcal{O}_i$ is transitive, too, and generated by \(\ell\)-cycles, hence it is primitive by Proposition~\ref{prop:coprime-primitive}.  If \(t > \ell\), it follows as above from Proposition~\ref{prop:primitive-and-cycles} and Corollary~\ref{lem:solitary-using-fixed-points} that \(\Omega\) is Gassmann solitary so it remains to consider the case \(\ell = t\).

  In that case, we obtain from \cite[Theorem~3.5B]{DixonMortimer} that either $N_1$ acts $2$-transitively on $\mathcal{O}_1$ or \(N_1 \le \operatorname{AGL}_1(\ell)\), meaning \(N_1\) has a normal cyclic subgroup of order \(\ell\) and acts like a group of affine transformations on \(\mathcal{O}_1\) under some identification of $\mathcal{O}_1$ with the finite field \(\mathbb{F}_\ell\).  In the former case, it is a consequence of the classfication of finite simple groups drawn by Feit in~\cite{Feit}*{Corollary~4.5} that if $\mathcal{O}_1$ is not Gassmann solitary as $N_1$-set, then \(\ell=11\) or \(\ell = \frac{q^k-1}{q-1}\) for a prime power~\(q\) and these possibilities are excluded in the Proposition.  So it remains to consider the latter case, in which \(N_1 \cong C_\ell \rtimes C_m\) is a semidirect product of cyclic groups with \(m \mid (\ell - 1)\) and \(C_m\) is the point stabilizer of the zero element under the identification of \(\mathcal{O}_1\) with \(\mathbb{F}_\ell\).  Since all subgroups of order \(m\) in \(N_1\) are conjugate  by the Schur-Zassenhaus theorem, the \(N_1\)-action on \(\mathcal{O}_1\) is Gassmann solitary.  But \(\mathcal{O}_1\) is an orbit of a normal subgroup of \(G\), hence it is a block.  So by Proposition~\ref{prop:solitary-using-blocks}, also the \(G\)-action on \(\Omega\) is Gassmann solitary.  
\end{proof}

\section{Proofs of the applications} \label{sec:conclusion}

In this section, we conclude Theorems~\ref{thm:prime-splitting} and~\ref{thm:geodesic} from the introduction.

\begin{proof}[Proof of Theorem~\ref{thm:prime-splitting}]
  Let \(K/\Q\) be a Galois extension containing both~\(k_1\) and~\(k_2\), set \(G = \operatorname{Gal}(K/\Q)\) and let \(H_1\) and \(H_2\) be the subgroups of \(G\) fixing \(k_1\) and \(k_2\) pointwise, respectively.  As we already pointed out in the introduction, the condition \(\zeta_{k_1} = \zeta_{k_2}\) is equivalent to \((G;H_1,H_2)\) being a Gassmann triple, and if some unramified rational prime \(p\) has one of the decomposition types \ref{item:twos}, \ref{item:not-divide}, or \ref{item:prime} in \(k_1\), then the Frobenius automorphism of any prime ideal \(\mathfrak{P}\) over \(p\) in \(K\) provides an element in \(G\) that acts with the corresponding cycle type on \(G/H_1\).  If \(k_1\) has a unique complex place, then after embedding \(K \subseteq \C\), complex conjugation yields an involution \(g \in G\), unique up to conjugation, which acts on \(G/H_1\) with cycle type \((1, \ldots, 1, 2)\), so condition~\ref{item:prime} of Theorem~\ref{thm:gassmann-triples} is satisfied.  Similarly, if \(k_1\) has a unique real place, we obtain a unique conjugacy class in \(G\) acting with cycle type \((1, 2, \ldots, 2)\) on \(G/H_1\), so condition~\ref{item:twos} of Theorem~\ref{thm:gassmann-triples} is satisfied.  In any case, Theorem~\ref{thm:gassmann-triples} shows that the Gassmann triple \((G; H_1, H_2)\) is trivial.  Hence \(H_1\) is conjugate to \(H_2\) and correspondingly \(k_1\) is isomorphic to \(k_2\).
\end{proof}

\begin{proof}[Proof of Theorem~\ref{thm:geodesic}.]
  Fix base points \(x_1 \in M_1\) and \(x_2 \in M_2\) over a common base point \(x_0 \in M_0\).  Let \(G = \pi_1(M_0, x_0)\) be the fundamental group and let \(H_1 = {p_1}_* \pi_1(M_1, x_1)\) and \(H_2 = {p_1}_* \pi_1(M_2,x_2)\) be the characteristic subgroups of \(G\) corresponding to the coverings~\(p_1\) and~\(p_2\).  Then
  \[ N = \bigcap_{g \in G}\, g^{-1} (H_1 \cap H_2) g \]
  is a finite index normal subgroup of~\(G\) and the corresponding finite sheeted regular covering
  \[ p_0 \colon (M, x) \rightarrow (M_0, x_0) \]
  has both \(p_1\) and \(p_2\) as intermediate covering.  Consider the triple \((G/N; H_1/N, H_2/N)\).  Since~\(p_1\) and~\(p_2\) are arithmetically equivalent, \cite{Sunada}*{Theorem~2} gives that it is a Gassmann triple.  Condition~\ref{item:twos}, \ref{item:not-divide}, or \ref{item:prime} effects by Theorem~\ref{thm:gassmann-triples} that the Gassmann triple is trivial.  So~\(p_1\) and~\(p_2\) are conjugate coverings and in particular~\(M_1\) is isometric to~\(M_2\).
\end{proof}

\end{document}